\documentclass[11pt, reqno]{amsart}

\usepackage{amsthm,amssymb,amstext,amscd,amsfonts,amsbsy,amsrefs,amsxtra,latexsym,amsmath,xcolor,mathrsfs,fancybox,upgreek, soul,url}
\usepackage[english]{babel}
\usepackage[all,cmtip]{xy}
\usepackage[utf8]{inputenc}
\usepackage[T1]{fontenc}
\usepackage{cancel}
\usepackage[draft]{hyperref}

\usepackage{comment}
\usepackage{mdframed}
\allowdisplaybreaks
\usepackage{mathtools}
\usepackage{enumerate}
\usepackage{thmtools}
\usepackage{thm-restate}
\usepackage{chngcntr}
\usepackage{enumitem}
\usepackage{etoolbox}
\usepackage{todonotes}
\usepackage{stmaryrd}
\usepackage{soul}
\hypersetup{final}

\usepackage{tikz}
\usepackage{verbatim}
\usetikzlibrary{quotes,angles}

\DeclarePairedDelimiter\abs{\lvert}{\rvert}
\DeclarePairedDelimiter\norm{\lVert}{\rVert}

\makeatletter
\let\oldabs\abs
\def\abs{\@ifstar{\oldabs}{\oldabs*}}
\let\oldnorm\norm
\def\norm{\@ifstar{\oldnorm}{\oldnorm*}}
\makeatother

\makeatletter
\glb@settings 
\makeatother

\usepackage[top=1.25in, bottom=1.25in, left=0.8in, right=0.8in]{geometry}

\newtheorem{theorem}{Theorem}
\newtheorem{lemma}[theorem]{Lemma}

\newtheorem{proposition}[theorem]{Proposition}
\newtheorem{conjecture}[theorem]{Conjecture}
\newtheorem{definition}[theorem]{Definition}

\theoremstyle{remark}

\newtheorem*{remark}{Remark}

\numberwithin{theorem}{section}
\numberwithin{proposition}{section}
\numberwithin{lemma}{section}
\numberwithin{corollary}{section}
\numberwithin{equation}{section}
\numberwithin{conjecture}{section}
\numberwithin{definition}{section}

\setlist[enumerate,1]{before=}
\AfterEndEnvironment{enumerate}{}

\newcommand{\N}{\mathbb{N}}

\newcommand{\R}{\mathbb{R}}

\usepackage{bm}

\renewcommand{\bmod}[1]{\ ( \mathrm{mod} \, #1 )}

\title{On the digits of the sum of proper divisors}

\author{K\"{u}bra Benl\.{i}, C\'{e}cile Dartyge, Charlotte Dombrowsky, Paul Pollack,\\ and Lola Thompson}

\address{Bo\u{g}az\.{i}\c{c}\.{i} University, Department of Mathematics, Bebek, 34342, \.{I}stanbul, T\"{u}rk\.{i}ye}
\email{kubra.benli@bogazici.edu.tr}

\address{Institut \'{E}lie Cartan  de Lorraine  \& Institut Universitaire de France, Universit\'{e} de Lorraine, BP 70239, 54506 Vand\oe uvre-l\`{e}s-Nancy Cedex, France}
\email{cecile.dartyge@univ-lorraine.fr}

\address{Universit\"{a}t Bielefeld, Fakult\"{a}t f\"{u}r Mathematik, Postfach 100131,
33501 Bielefeld, Germany}
\email{cdombrow@math.uni-bielefeld.de}

\address{Department of Mathematics, University of Georgia, 
Athens, GA 30602, United States}
\email{pollack@uga.edu}

\address{Mathematics Institute, Utrecht University, Hans Freudenthalgebouw, Budapestlaan 6, 3584 CD Utrecht, The Netherlands}
\email{l.thompson@uu.nl}

\begin{document}

\begin{abstract}
    We study several probabilistic questions concerning the digits of $s(n)$, the sum of proper divisors of an integer $n$. In particular, we show that $s(n)$ obeys Benford’s law with respect to logarithmic density. Moreover, we show that, for every function $k(x) \rightarrow \infty$, almost all integers $n \leq x$ have every decimal digit occurring among the first $k(x)$ digits and the last $k(x)$ digits of $s(n).$
    
    We also present an upper bound for the number of composite integers $n$ up to $x$ for which $s(n)$ is missing at least one digit in its decimal expansion. This is in contrast with the main result of a recent paper of Benli, Cesana, Dartyge, Dombrowsky, and Thompson, in which the inputs $n$ were not required to be composite. It turns out that the primes make a substantial contribution to the preimage set $s^{-1}(\mathcal{A})$, where $\mathcal{A}$ is a set of integers with missing digits. Our result for composite $n$ shows that the count is much smaller when prime inputs are excluded. 
  
\end{abstract}

\maketitle

\section{Introduction}

Let $s(n)$ denote the sum of proper divisors of a positive integer $n$. The function $s(n)$ has a long and rich history. Pythagoras first observed that $s(6) = 6$ and $s(28) = 28.$ He called such numbers \textit{perfect}. The study of perfect numbers, that is, integers $n$ for which $s(n) = n$, has played an important role in number theory. It has even been remarked \cite{pomerance} that perhaps $s(n)$ is the ``first function'' to ever have been studied in mathematics. 

The goal of this paper is to study the digits of $s(n)$ from a probabilistic standpoint. Observe that, if $n = 2^{5865}$, then $s(n) = 2^{5865} - 1.$ The last $10$ decimal digits of $s(n)$ are $5046879231$, so every decimal digit appears. It is natural to wonder how common this phenomenon is. We show that it is actually very common, provided that we allow the number of ``last'' decimal digits to be sufficiently large. We also show that the ``first'' decimal digits have the same property. In fact, we establish these results for an arbitrary base $g\ge 2$.

\begin{theorem}\label{thm:firstlastk} Fix  $g\ge 2$.
 If $k(x) \to \infty$ as $x\to \infty$, then asymptotically $100\%$ of integers $n\leq x$ are such that $s(n)$ has all $g$ digits in base $g$ appearing among its last (i.e., the rightmost) $k(x)$  digits. Moreover, the same result holds for the first (i.e., the leftmost) $k(x)$ digits.
\end{theorem}

This suggests that the base $g$ expansion of $s(n)$ behaves in many respects like that of a random integer of comparable size. 

Theorem \ref{thm:firstlastk} concerns the occurrence of digits in prescribed positions of the decimal expansion of $s(n)$. A complementary question is to determine how the leading digits are distributed. The second goal of this paper is to show that the digits of $s(n)$ obey Benford's law with respect to logarithmic density.

Fix a base $g\ge 2$. We say that ``Benford's law'' (in base $g$) holds for a sequence of positive real numbers if, for each positive integer $D$, we have $D$ appearing as a block of leading digits of the terms of the sequence with limiting frequency $P(D):= \log_{g}(D+1) - \log_{g}(D) = \log_{g}(1 + \frac{1}{D})$. For example, if a sequence obeys Benford's law in base $10$, then $1$ appears as the leading digit with limiting frequency $P(1) = \log_{10}(2) = 30.1\%$, $7$ appears with frequency $P(7) = \log_{10}(8/7) = 5.1\%$, and the leading digits $2026$ appear with frequency $\log_{10}(2027/2026) = 0.021\%$. This version of Benford's law is called ``strong Benford'' by Diaconis \cite{diaconis77}, but we will simply use the term ``Benford'' here, as it is the only version of Benford's law to be discussed. 

\begin{theorem}\label{th:Benford} Fix $g\ge 2$. The function $s(n)$ satisfies Benford's law with respect to logarithmic density. In other words, for each positive integer $D$, $$\delta_{\log} \{n>1: \text{the leftmost base-$g$ digits of $s(n)$ are the digits of $D$}\} = \log_{g} \left(1 + \frac{1}{D}\right).$$
\end{theorem}

As we show at the end of Section \ref{sec:Benford}, $s(n)$ does \emph{not} satisfy Benford's law with respect to natural density.

In Theorem \ref{thm:firstlastk}, we observe that $s(n)$ usually contains all possible digits. Both Theorems \ref{thm:firstlastk} and \ref{th:Benford} suggest that the digits of $s(n)$ exhibit a high degree of randomness. We now turn our attention to the values of $s(n)$ whose base-$g$ expansions are unusually restrictive, namely those missing one or more digits. Understanding how frequently such values occur leads us to study the preimages of sparse sets under $s$.

One surprising fact about $s(n)$ is that it can map sets with asymptotic density zero to sets with positive asymptotic density. The preimages of $s(n)$ can also display surprising behavior. For example, it is possible for a set $\mathcal{A}$ to have positive asymptotic density while $s^{-1}(\mathcal{A})$ has asymptotic density $0$. Erd\H{o}s even showed that there are sets $\mathcal{A}$ with positive asymptotic density for which $s^{-1}(\mathcal{A})$ is empty!

Some behavior of $s(n)$ is still not well-understood. There is a conjecture of Erd\H{o}s, Granville, Pomerance, and Spiro \cite{egps} from 1992 which can be stated as follows:

\begin{conjecture} Let $\mathcal{A}$ be a set of integers with asymptotic density zero. Then $s^{-1}(\mathcal{A})$ also has asymptotic density zero.    
\end{conjecture}

We shall refer to this conjecture as the ``EGPS Conjecture'' throughout this paper. The conjecture remains open, although it has been proven in some special cases. For example, if $\mathcal{A}$ is the set of prime numbers \cite{Pollack2014}, or the set of palindromes in any given base \cite{PollackPalindromes}, or the set of sums of two squares \cite{TroupeSumofsquares}, it has been shown that the EGPS Conjecture holds. In all of these cases, the elements in the set have very specific structural features. It has also been shown in \cite{PollackPomeranceThompson} that, without imposing any special structure on the elements of the set, the EGPS Conjecture holds for sets $\mathcal{A}$ with size $O(x^{1/2 + \varepsilon})$ where $\varepsilon$ is a fixed function tending to $0$ as $x \rightarrow \infty$. 

In a previous work, Benl\.{i}, Cesana, Dartyge, Dombrowsky, and Thompson  confirmed that the EGPS conjecture holds for sets of integers with missing digits. In particular, they showed:

\begin{theorem} \cite[Theorem 1.8]{bcddt}\label{short theorem in intro}\label{thm: bcddt} 
Fix $g\geq 2$, $\gamma\in(0,1)$, and a nonempty set $\mathcal{D}\subsetneq \{0,1,\dots, g-1\}$. For all sufficiently large $x$, the number of $n\leq x$ for which $s(n)$ has all of its digits in base $g$ restricted to digits in $\mathcal{D}$ is $O(x\exp(-(\log\log x)^\gamma)).$ 
\end{theorem}

The case $g = 2$, omitted from the statement in \cite{bcddt}, also follows by a separate elementary argument. See Appendix \ref{appendix}. 

Observe that $s(p) = 1$ for all primes $p$, which means that whenever the set $\mathcal{D}$ contains $1$, the size of the preimage set of $\mathcal{A}$ has $\pi(x) \sim x/\log x$ as a lower bound. Thus, the upper bound given by Theorem \ref{short theorem in intro} is essentially ``best possible'' because we cannot replace the constant $\gamma \in (0,1)$ with any constant strictly greater than $1$.

In the present paper, we explore what happens to this upper bound when we exclude the contribution from prime inputs. At first glance, one might expect the prime inputs to play only a negligible role in questions concerning the digits of $s(n)$. Surprisingly, this is not the case. When we consider only composite values of $n$, we show that the count is considerably smaller:

    \begin{theorem}\label{theorem-general}
      Let $g\in\N$, $g\ge 2$ and $a_0\in\{ 1,\ldots , g-1\}$.   There exists a constant $c=c(g)>0$, depending on $g$ such that 
    $$\#\left\{n \leqslant x: n \ \mathrm{composite}, s(n) \text { has no } a_0 \text { in its base-$g$ expansion}\right\}\ll x \exp (-c \sqrt{\log x}).$$   
    \end{theorem}

\subsection{Notation}\label{Notation} We record some notation that will be used throughout the paper. For $n\in\mathbb{N}$, $n\geq 2$, let $P^{+}(n)$ denote the largest prime factor of $n$ and $P^{-}(n)$ denote the smallest prime factor of $n$.

    Fix $g \in \N$. Let $ \mathcal{D} \subsetneq \{ 0,..,g-1\}$. Then the set of integers with missing digits is given by:
     \begin{align*}\label{def:WD}
 \mathcal{W}_{\mathcal{D}}\coloneqq \left\{ n\in\mathbb{N} :  n= \sum_{j=0}^J\varepsilon_j(n) g^j, J \geq 0,  \varepsilon_j(n)\in\mathcal{D}, \varepsilon_J \neq 0\right\}. 
 \end{align*}
We denote by $\mathcal{W}_{\mathcal{D}}(g^k)$ the subset of  $\mathcal{W}_{\mathcal{D}}$ of integers up to $g^k$:
\begin{align*}
    \mathcal{W}_{\mathcal{D}}(g^k):= \{ n \in \mathcal{W}_{\mathcal{D}}: n \leq g^k\}.
\end{align*}

In Section \ref{sec:Benford}, we will state Benford's law in terms of logarithmic density, which we define as follows:

\begin{definition} For $A \subseteq \mathbb{N}$, we define the logarithmic density of $A$ as follows:

$$\delta_{\log}(A) = \lim_{x \rightarrow \infty} \frac{1}{\log x} \sum_{\substack{n \leq x \\ n \in A}} \frac{1}{n},$$ provided that the limit exists.
    
\end{definition}

\section{Digital hits}

The goal of this section is to prove Theorem \ref{thm:firstlastk}, which has been separated into two parts. The first part is stated below as Theorem \ref{thm:lastk}, and the second part is stated as Theorem \ref{thm:firstk}.

We will require the following lemma in our first proof:

\begin{lemma}\cite[Lemma 2.1]{Pollack2014}\label{lemma:oldpaper}
		Let $x\geq 3$. Let $q$ be a positive integer. Then 
		$$\sum_{\substack{{n\le x}\\ {q \nmid \sigma (n)}}}1\ll \frac{x}{(\log x)^{1/\varphi(q)}},$$
		uniformly in $q$ where $\varphi$ is Euler's totient function.
        
	\end{lemma} 

Now we are able to prove Theorem \ref{thm:lastk}. 

	\begin{theorem}\label{thm:lastk}
 If $k(x) \to \infty$ as $x\to \infty$, then asymptotically $100\%$ of integers $n\leq x$ are such that $s(n)$ has all $g$ digits in base  $g$ appearing among its last (i.e., the rightmost) $k(x)$ digits. 

	\end{theorem}

\begin{proof}

We first note that the case when $g=2$ is handled in Appendix \ref{appendix}, as the statement of the theorem boils down to showing that the number of $n\leq x$ for which all digits of $s(n)$ are 1 is $o(x)$ in this case.  Let $g\geq 3$. Let $x$ be large, assume $k=k(x)\to \infty$ as $x\to \infty$. Let $\gamma<1$. Note that $k(x)\le \frac{\log x}{\log g}$. Without loss of generality, we can assume $k\leq (\log \log\log x)^\gamma$; otherwise, we can simply replace $k(x)$ with $\min\{k(x),(\log\log\log x)^\gamma\}$. Then we would like to show that 
    for any subset 
    $\mathcal{D}\subsetneq\{ 0,\ldots, g-1\}$, 
	$$\#A:=\#\{n\leq x: s(n)\equiv B\bmod {g}^k, \, \text{for some $B\in \mathcal{W}_{\mathcal{D}}(g^{k}-1)$}\}=o(x).$$

    Following the proof of Theorem 1.8 in \cite{bcddt}, for $n\in A$, if  $g^k\mid \sigma(n)$, we have 
	\begin{align*} 
		n= \sigma(n)-s(n) \equiv -s(n) \equiv -B\bmod{g^k}
	\end{align*}
	so that 
	\begin{align*}
	 \sum_{\substack{n\in A \\ \sigma(n)\equiv 0\bmod{g^k}}} 1  \leq \sum_{\substack{n\leq x\\ n\equiv -B\bmod{g^k} \\ B\in \mathcal{W}_{ \mathcal{D}}(g^{k}-1)}} 1 &= \sum_{B\in \mathcal{W}_{ \mathcal{D}}(g^{k}-1)} \sum_{\substack{n\leq x\\ n\equiv -B\bmod{g^k}}} 1 \\ &\leq |\mathcal{D}|^k \left(\frac{x}{g^k}\right)+|\mathcal{D}|^k \\ &\ll x  \exp \left( -k \log \left(g/|\mathcal{D}|\right)\right) +|\mathcal{D}|^k.
	\end{align*}
	The $|\mathcal{D}|^k$ does not pose a problem since $|\mathcal{D}|^k \leq (g-1)^k \leq \exp(\log(g-1) (\log \log \log x)^\gamma) = (\log \log x)^{o(1)} = o(x)$. Consequently, the last expression in the displayed equation is $o(x)$ if $k(x)\to \infty$ as $x\to \infty$ and since $k\ll (\log\log\log x)^\gamma$. 
    
    To complete the proof, we need to show that the number of $n\leq x$  with $g^k \nmid \sigma (n)$ is $o(x).$ To accomplish this, we appeal to Lemma \ref{lemma:oldpaper}, taking $q = g^k$. Since  $k\leq (\log\log\log x)^\gamma$, we then have  the number of $n\leq x$  with $g^k \nmid \sigma (n)$ is $\ll x\exp\left(-\frac{\log\log x}{\varphi(g^k)}\right)\leq  x\exp\left(-\frac{\log\log x}{g^k}\right)\leq x\exp\left(-\frac{\log\log x}{\exp((\log g) (\log \log \log x)^\gamma)}\right)=o(x).$

\end{proof}

Again fix $g\ge 2$. We now present the proof that the \textit{first} $k(x)$ base $g$ digits of $s(n)$ include all $g$ digits $100\%$ of the time. We first collect some auxiliary results. 

Let $M = \prod_{p \leq z} p^{z}$, where $z = g^{10k}$.	By the Prime Number Theorem, $M \leq 3^{z^2}$. Let $n'\in (0,M]$ be the least positive representative of $n\bmod{M}$ and $n_{z}$ the $z$-smooth (or $z$-friable) part of $n'$, that is $n_z=\prod_{\substack{p^e\mid\mid n'\\p\leq z}}p^e$. We have the following.	
\begin{lemma}\label{lem: n to nz}
	Let $x$ be large and $k\in \mathbb{N}$ with $k(x)\to\infty$ as $x\to \infty$. Put $z = g^{10k}$.
	For all but $o(x)$ values of $n \leq x$, one has 
	$0 \leq s(n)/n - s(n_z)/n_z < z^{-1/2}$
	and $s(n_z)/n_z \geq 1/k.$
	\end{lemma}
	\begin{proof}
First note that	if $n_z$ is divisible by $p^{z}$ for some prime $p \leq z$, then so is $n'$ and since $n\equiv n'\bmod M$, so is $n$. Thus, the number of $n\leq x$ such that $n_z$ is divisible by $p^{z}$ for some prime $p \leq z$ is at most $x\sum_{p\leq z} 1/p^{z}\ll x/2^z=o(x)$.

Thus, we can assume $n_z$ is not  divisible by $p^{z}$ for any prime $p \leq z$. It follows that $n_z$, the $z$-smooth part of $n'$, is also the $z$-smooth part of $n$. 
So, for such $n$ and $n_z$, we have 
	$$0 \leq \frac{s(n)}{n} - \frac{s(n_z)}{n_z} = \sum_{\substack{d | n\\ P^+(d) >z}} \frac{1}{d} \leq \sum_{\substack{d\mid n\\ d > z}} \frac{1}{d} .$$
	Averaging over positive integers $l\leq x$ allows us to obtain 
    $$\sum_{l\leq x}\sum_{\substack{d | l\\ d > z}} \frac{1}{d} \ll \frac{x}{z}.$$ 
    Applying  Markov's inequality yields
	\[\#\left\{n\leq x: \frac{s(n)}{n} - \frac{s(n_z)}{n_z}\geq z^{-1/2}\right\}\ll z^{1/2}x/z=o(x).\]

	Now, if $ p\mid n_z$, then $\frac{s(n_z)}{n_z} = \sum_{\substack{d | n_z\\ d> 1}} 1/d \geq  1/p$. So,  $\frac{s(n_z)}{n_z} \geq  1/k$ unless $n_z$ is $k$-rough. But then $n$ is also $k$-rough, and the number of such $n\leq x$ is $\ll x/\log k=o(x)$.  
	\end{proof}

\begin{remark} We briefly remark on certain exceptional cases that we can safely ignore in our proof of Theorem \ref{thm:firstk} below, since they contribute only $o(x)$ to the total count. As always, we assume that $k(x)\to \infty$ as $x\to \infty$. Then 
\begin{itemize}
    \item If $n$ has $r$ decimal digits, we can assume that $g^{r} > x/g^{\log \log k}$.  Otherwise, if $g^r \leq x/g^{\log \log k}$, then $n\leq g^{r} \leq x/g^{\log \log k}.$ The number of these $n$ is $o(x)$. 
    
    \item If $s(n)$ has $m$ digits base $g$ then we can assume that $|m-r| < \log \log k$. This will follow from Lemma \ref{lem:m-r}, which we prove below.

\end{itemize}

\end{remark}

\begin{lemma}\label{lem:m-r} Let $g\geq 2$, $g\in \mathbb{N}$. The number of $n\leq x$ 
    with the number $m$ of base $g$ digits of $s(n)$, and the number $r$ of base $g$ digits of $n$, satisfies $|m-r| \geq \log \log k$ is $o(x)$.
    
\end{lemma}

    \begin{proof}
If $n\leq x$ has $r$ digits, then $g^{r-1}\leq n< g^r$. And if $s(n)$ has $m$ digits then $g^{m-1}\
 \leq s(n)< g^{m}$. So, $g^{m-r-1}< \frac{s(n)}{n}< g^{m-r+1}$. If $|m-r| \ge \log\log k$,  either $m-r \ge \log\log{k}$ or $r-m \ge \log\log{k}$. In the case $m-r\ge \log\log{k}$, we have $\frac{s(n)}{n}>\frac{1}{g}(\log k)^{\log g}$. On the other hand, if $r-m\ge \log\log{k}$, then $\frac{s(n)}{n} < \frac{g}{(\log{k})^{\log{g}}}$.  
 
Davenport \cite{davenport33} has shown that for each fixed $u\ge 0$, the density of $n$ with $s(n)\le un$ exists. Moreover, calling this function $D(u)$, we have that $D(u)$ is continuous, that $D(u) \to 0$ as $u\to  0$, and that $D(u)\to 1$ as $u\to\infty$. For any fixed $M>0$, we have that $(\log{k})^{\log{g}}/g$ eventually exceeds $M$ while $g/(\log{k})^{\log{g}}$ is eventually smaller than $1/M$. Hence,
\begin{multline*} \limsup_{x\to\infty}\frac1x\left(\#\left\{n\leq  x: \frac{s(n)}{n}>\frac{1}{g}(\log k)^{\log g}\right\}+\# \left\{n\leq x: \frac{s(n)}{n}<\frac{g}{(\log k)^{\log g}}\right\}\right)\\
\leq (1-D(M))+D(1/M).
\end{multline*}
The lemma follows upon sending $M$ to infinity.\end{proof}

	We are now ready to prove the second half of Theorem \ref{thm:firstlastk}. 
	
	\begin{theorem}\label{thm:firstk}
	Assume $k(x)\to \infty$ as $x\to \infty$. 	
	Then the number of $n\leq x$ where, in the base $g$ expansion of $s(n)$, not all digits are represented among the first $k$ digits is $o(x)$. In other words, asymptotically $100\%$ of integers $n\leq x$ are such that $s(n)$ has all $g$ digits in base $g$ appearing among its first (i.e., leftmost) $k(x)$ digits.
	\end{theorem}
	\begin{proof}
    We can assume our function $k=k(x) \to \infty$ grows very slowly, e.g., $k(x) \leq \log\log\log x$ (otherwise, we can simply replace $k(x)$ with $\min\{k(x),\log\log\log x\}$).
     
Let $n\leq x$ be such that in the base $g$ expansion of $s(n)$, not all digits are represented among the first $k$ digits. Let $r\in \mathbb{N}$ be such that \[ g^{r-1}\leq n<g^r. \] By our first remark above, we may assume that \[ g^r>\frac{x}{g^{\log\log k}}.\] Let $m\in \mathbb{N}$ be such that  \[ g^{m-1}\leq s(n)<g^m.\] By our second remark above, we may assume  that 
\[ |m-r| < \log \log k. \] 

Write $s(n) = g^{m-k} L  + T$, where $L, T\in \mathbb{N}$ with $0\leq T< g^{m-k}$. Then $L < g^{k}$, and the digits of $L$ are the $k$ leading digits of $s(n)$. Note that the number of possibilities for $L$ is $\ll g^{ck}$, for a constant $c < 1$ depending only on $g$ (because $L$ is missing at least one base $g$ digit). This puts $s(n)$ in the interval  \begin{align*} I_L&:=[L\cdot g^{m-k},(L+1)\cdot g^{m-k}) \\&:=[t_L,T_L). \end{align*}

We fix $r$, $m$, and $L$, and count the number of corresponding $n$. It is convenient to sort these $n$ into residue classes modulo $M = \prod_{p \leq z} p^{z}$, where $z = g^{10k}$. By the Prime Number Theorem, $M \leq 3^{z^2} =x^{o(1)}$, where the last equality comes from our restriction on the size of $k$. Let  $n'\in (0,M]$ be the least positive representative of $n \bmod M$  and $n_{z}$ be the $z$-smooth part of $n'$. Excluding the $o(x)$ exceptional values of $n$ described in Lemma \ref{lem: n to nz} above, we have $$\frac{s(n)}{n} \leq \frac{s(n_z)}{n_z} \left(1 +\frac{n_z}{s(n_z)}  z^{-1/2}\right) \leq {s(n_z)} (1 + k z^{-1/2}).$$ 
    
We now fix also $n'\in \mathbb{N}\cap (0,M]$ and consider only nonexceptional integers $n \equiv n' \bmod M$ corresponding to our  choices of $r$,$m$, and $L$ from above.

If $s(n) = n \cdot s(n)/n$ is in one of the intervals $I_L$ from above, then $n$ is in $(n/s(n)) I_L$. Notice that $s(n)/n \geq s(n_z)/n_z$, and so
	$n/s(n)\leq  n_z/s(n_z)$. Since $T_L$ is the upper endpoint of $I_L$, the dilated interval $(n/s(n))I_L$ has upper endpoint at most \[ T_{L'} := (n_z/s(n_z)) T_L.\]
	
	Now, we consider the lower endpoint. We have
    
	$$\frac{s(n)}{n} \leq \frac{s(n_z)}{n_z}+ z^{-1/2}
	= \frac{s(n_z)}{n_z}\left(1 + z^{-1/2} \frac{n_z}{s(n_z)}\right)
\leq \frac{s(n_z)}{n_z} (1 + z^{-1/3})$$
since $z=g^{10k}.$ Hence,
 \[ \frac{n}{s(n)} \geq \frac{n_z}{s(n_z)} (1-z^{-1/3}).\] Thus, if $t_L$ is the lower endpoint of $I_L$, then $(n/s(n))I_L$ has lower endpoint at least \[ t_{L'} := n_z/s(n_z) (1-z^{-1/3}) t_L.\] 
 
Every nonexceptional $n\equiv n'\bmod{M}$ corresponding to our  choices of $r$,$m$,  and $L$ belongs to the interval $[t_{L'}, T_{L'})$. The length of this interval is
$$T_{L'}-t_{L'}=\frac{n_z}{s(n_z)} T_L-\frac{n_z}{s(n_z)}(1-z^{-1/3}) t_L=\frac{n_z}{s(n_z)} |I_L| + z^{-1/3} \frac{n_z}{s(n_z)} t_L,$$
and so
\begin{align*}
\#\{n\leq x: n \equiv n' \bmod M,\,\, t_{L'}\leq n< T_{L'}    \}&\leq \frac{1}{M}\frac{n_z}{s(n_z)} \left(|I_L| + z^{-1/3} t_L\right)+1
\\&\leq 1 + \frac{k g^{m-k}(1+z^{-1/3}L)}{M}.
\end{align*}

We bound the total number of nonexceptional $n$ by summing over parameters previously held fixed. First, we sum over the $M$ possible values of $n' \in (0,M]$. 
Recalling that $z=g^{10k}$ and $L\ll g^{ck}$ for $c<1$, the number of $n$ arising, for a fixed choice of $r, m$, and $L$, is at most
    \[ M + 2k \cdot g^{m-k} \le 3k g^{m-k}. \]
For the last inequality, we used that $M=x^{o(1)}$ while (for large $x$)
\[ g^{m-k} = g^{r} g^{m-r} g^{-k} \gg \frac{x}{g^{\log\log{k}}} \cdot \frac{1}{g^{\log\log{k}}} \cdot \frac{1}{g^k} > \frac{x}{g^{2k}} \ge \frac{x}{(\log\log{x})^{O(1)}}.\]

Next, we sum on the possibilities for $L$; this gives an upper bound of the form $3k \cdot g^{-c' k} g^m$, for a positive $c'>0$, bounded away from $0$. Furthermore, $m \leq r + \log \log k$. Summing $3k \cdot g^{-c' k} g^m$ on these $m$, we obtain a quantity bounded above by $g^r g^{-c'' k}$, for another small $c'' > 0$. Summing on the possibilities for $r$ gives an upper bound of $x g^{-c''' k}$ for $c'''>0$. If we combine this upper bound with the previously mentioned $o(x)$ contributions from the exceptions, we still obtain $o(x)$ for our final count.
		
	\end{proof}

\section{Benford's law}\label{sec:Benford}

It was shown by Diaconis \cite[Theorem 1]{diaconis77} that a sequence $\{a_n\}$ of positive integers obeys Benford's law, in base $10$, precisely when $\{\log_{10} a_n\}$ is uniformly distributed modulo $1$. Diaconis's result is about decimal expansions and natural density, but the argument works just as well in any base $g$ and for logarithmic density.

\begin{proposition} A sequence $a_n > 0$ is base $g$ Benford in the sense of logarithmic density if and only if $\{a_n\}$ is uniformly distributed mod $1$, in the sense of logarithmic density. That is, for every interval $I \subseteq [0, 1)$ we have $\delta_{\log}(\{n : \{\log_{g} a_n\} \in I\}) = |I|.$    
\end{proposition}

 The following criterion (see, for example, \cite[p. 209]{gonggulupollack}) will allow us to decide uniform distribution.
    
	\begin{theorem}[Weyl's criterion for logarithmic densities]\label{thm:Weylscrit}
 The sequence $\{a_n\}$ is uniformly distributed modulo $1$ with respect to logarithmic density 
if and only if 
$$\lim_{N\to \infty}\frac{1}{\log N}\sum_{n\leq N}\frac{e^{2\pi i \ell a_n}}{n}=0$$ for all integers $\ell\neq  0$.
	\end{theorem}
	
We define $$\{a_k\}_{k\geq 2}:=\{\log_{g} s(k)\}_{k\geq 2}.$$
	By Theorem \ref{thm:Weylscrit}, in order for $s(n)$ to obey Benford's law, it is enough if $e^{2\pi i \ell \log s(n)/\log g} $
has logarithmic mean $0$ (over integers $n\geq 2$) for every fixed integer $\ell\neq 0$. 

That certainly holds if $s(n)^{i\alpha}$ has logarithmic mean  $0$ (over integers $n\geq 2$) for each fixed nonzero real number $\alpha$. We  prove this by appealing to the following weighted version of Hal\'{a}sz's Theorem:

	\begin{proposition}[Proposition 1.2.6, \cite{GranvilleSoundararajan}]\label{Prop GS}
    If $f$ is a multiplicative function with $|f(n)| \leq 1$ for all $n\in\mathbb{N}$, then
\[\frac{1}{\log x}\left|\sum_{n\leq x}\frac{f(n)}{n}\right|\leq\exp\left(-\frac12\sum_{\substack{q\leq x\\q\text{ prime}}}\frac{1-\Re\left(f(q)\right)}{q}\right).\]
	\end{proposition}

	Proposition \ref{Prop GS} tells us that, if $f$ is a multiplicative function of modulus at most $1$, then $f$ has logarithmic mean value $0$ unless ``$f$ pretends to be 1.'' This will come in handy when we prove our main result below. 
    
    We will need the following definition:

    \begin{definition}\label{def:specialprimedivisor} We say that $p^e$ is the \textit{special prime divisor} of $n$ if it is the prime power $p^e$ for which  $p$ is the least prime divisor of $n$ and $p^e\mid\mid n$. We denote the special prime divisor of $n$ by $\gamma(n)$. \end{definition}

    Now that we have all of the ingredients, we are ready to prove that $s(n)^{i\alpha}$ has logarithmic mean $0$.

\begin{theorem}\label{thm:slogmean} For any fixed nonzero $\alpha\in\R$, we have $$\lim_{N \rightarrow \infty} \frac{1}{\log N} \sum_{1<n \leq N} \frac{s(n)^{i\alpha}}{n} = 0.$$
In particular,   for any integer base $g\geq 2$, $s(n)$ obeys Benford's law in terms of logarithmic density in base $g$.
\end{theorem}

\begin{proof} We cannot directly apply Proposition \ref{Prop GS} to estimate the mean value of $s(n)^{i\alpha}$, since $s(n)$ is not multiplicative. To work around this, write $s(n) = \sigma(n) - n = \sigma(n) (1-n/\sigma(n))$, and let \[ u(n):= n/\sigma(n).\] Since $0 < u(n)<1$ for all $n>1$,  Newton's binomial theorem gives us that
\begin{align*} s(n)^{i \alpha} &= \sigma(n)^{i\alpha}(1 - u(n))^{i\alpha} \\ &= \sum_{k \geq 0} \binom{i\alpha}{k} \sigma(n)^{i\alpha} (-u(n))^k.\end{align*}

If the sum on $k$ were finite, it would suffice by linearity to show that each term $\sigma(n)^{i\alpha} u(n)^k$ has logarithmic mean value $0$. That mean value result will be shown below using Proposition \ref{Prop GS}. However, since this is an infinite sum, we must be more careful. We truncate the binomial series at a finite level, far enough out to make the error small. This would be straightforward if we had $u(n) \leq a$ for some fixed $a< 1$. However, $u(n)$ can be arbitrarily close to $1$, and this necessitates partitioning our set of integers $n$ according to their special prime divisors (as in Definition \ref{def:specialprimedivisor}).

Each $n\in\mathbb{N}$ with special prime divisor $\gamma (n)=p^e$ has the form $n = p^em$, where the least prime dividing $m$ is greater than $p$. For each of these $n$, we have $\sigma(n)/n\geq \sigma(p^e)/p^e$, and so $u(n) = n/\sigma(n) \leq p^e/\sigma(p^e) < 1$. Since $u(n)$ is bounded away from $1$, when we look at the series expansion
	
$$	s(n)^{i\alpha} = \sum_{k \geq 0} \binom{i\alpha}{k}\sigma(n)^{i\alpha} (-u(n))^{k},$$ 
we will know where to cut off the righthand side for an acceptable error. 
 
Let $B$ be a fixed real number, which we will later choose to be very large.
We split the sum 
\[
T(N):=\frac{1}{\log N}
\sum_{1< n\le N}
\frac{(s(n))^{i\alpha}}{n}
\]
in two sub-sums $T(N)=T_1 (N)+T_2 (N)$ where 
\[T_1 (N)
:= \frac{1}{\log N}\sum_{\substack{{1<n\leq N}\\
{\gamma (n)\le B}}}\frac{1}{n}\sum_{k \geq 0} \binom{i\alpha}{k}\sigma(n)^{i\alpha} (-u(n))^{k}\]

and 
\[T_2 (N)
:= \frac{1}{\log N}\sum_{\substack{{1<n\leq N}\\
{\gamma (n)> B}}}\frac{1}{n}\sum_{k \geq 0} \binom{i\alpha}{k}\sigma(n)^{i\alpha} (-u(n))^{k}.\]

\subsection{Bounding the sum $T_1 (N)$.}

 We define
$$\rho := \max_{p^e\le B}
\frac{p^e}{\sigma(p^e)}.$$
If $n$ has a special prime power divisor $p^e\le B$, then $$u(n) = \frac{n}{\sigma(n)} \leq \rho < 1.$$ Hence, the binomial series $$(1 - u(n))^{i\alpha} = \sum_{k \geq 0} \binom{i\alpha}{k}(-u(n))^k$$ converges uniformly on the set of $n$ with a special prime divisor $\le B$. We fix $K = K(B, \alpha)$ sufficiently large so that $$\sup_{0\leq u \leq \rho}\left |(1 - u)^{i\alpha} - \sum_{k=0}^K \binom{i\alpha}{k}(-u)^k\right| < \frac{1}{B}.$$ Since $|\sigma(n)^{i\alpha}| =1$, we then have uniformly for $n$ with a special prime divisor $\le B$, that $$\left|s(n)^{i\alpha} - \sum_{k = 0}^K \binom{i\alpha}{k}(-1)^k\sigma(n)^{i\alpha}u(n)^k\right| < \frac{1}{B}.$$ 

So, we rewrite $T_1(N)$ as 
\begin{align*}
 T_1 (N)&= \frac{1}{\log N}\sum_{p^e\leq B}\sum_{\substack{1<n\leq N\\
 \gamma (n)=p^e }}\frac{1}{n}\sum_{0\leq k \leq K} \binom{i\alpha}{k}\sigma(n)^{i\alpha} (-u(n))^{k} \\&+ \frac{1}{\log N}\sum_{\substack{1<n\leq N\\
 \gamma (n)\le B }}\frac{1}{n}\sum_{k > K} \binom{i\alpha}{k}\sigma(n)^{i\alpha} (-u(n))^{k}
    \end{align*}
and by our choice of $K$,
\begin{align*}
   \left|\frac{1}{\log{N}}\sum_{\substack{1<n\leq N\\
   \gamma (n)\leqslant B}}\frac{1}{n}\sum_{k > K} \binom{i\alpha}{k}\sigma(n)^{i\alpha} (-u(n))^{k}\right|\leq& \frac{1}{\log{N}} \sum_{\substack{1<n\leq N\\
   \gamma (n)\leqslant B}}\frac{1}{n}\left|\sum_{k > K} \sigma(n)^{i\alpha}\binom{i\alpha}{k}(-u(n))^{k}\right|\\\le &\frac{1}{B \log{N}}\sum_{\substack{1<n\leq N\\
   \gamma (n) \leqslant B}} \frac{1}{n} \le \frac{2}{B},
\end{align*}
which will be acceptable to us.

It remains to handle the contribution to $T_1(N)$ from
\begin{equation}\label{eq:tobeohof1}\frac{1}{\log N}\sum_{p^e\leq B}\sum_{\substack{1<n\leq N\\
 \gamma (n)=p^e }}\frac{1}{n}\sum_{0\leq k \leq K} \binom{i\alpha}{k}\sigma(n)^{i\alpha} (-u(n))^{k}.\end{equation}
We change the order of summation in this finite sum and attempt to compute the mean value of one of the terms $\sigma(n)^{i\alpha} (-u(n))^{k}$ over $n$ with $\gamma(n)\leq B$. We have
 
	\begin{align}\label{decompose sum}
\nonumber	
\sum_{\substack{{n\leq N}\\ {\gamma (n)\leqslant B}}}\frac1n\sigma(n)^{i\alpha} (-u(n))^k  =& \sum_{p^e\leq B}\frac{1}{p^e}\sigma(p^e)^{i \alpha} (-u(p^e))^{k}\sum_{\substack{n\leq N\\\gamma (n)=p^e\\
n=p^em}} \frac1m \mathbf{1}_{q\mid m,\  q\, \text{prime} \implies q>p } \sigma(m)^{i \alpha} (u(m))^{k}
    \\=&\sum_{p^e\leq B}\frac{1}{p^e}\sigma(p^e)^{i \alpha} (-u(p^e))^{k}\sum_{m\leq \frac{N}{p^e}} \frac1m \mathbf{1}_{\substack{q\mid m, \,q \ \text{prime} \implies q>p }}  \sigma(m)^{i \alpha} (u(m))^{k}.
    \end{align}
    
Let $b_p(m)$ be the $1$-bounded multiplicative function defined by $$b_{p}(m):=\mathbf{1}_{q\mid m,\, q \text{ prime} \implies q>p } \sigma(m)^{i \alpha} (u(m))^{k}.$$ We will show that each of the functions $b_p(m)$ has logarithmic mean $0$. Hence, for each fixed $p^e$, the inner sum in \eqref{decompose sum} is of size $o(\log{N})$. Since $B$ is fixed, the triangle inequality then gives 
\[ \frac{1}{\log{N}} \sum_{p^e \le B} \frac{1}{p^e}\sigma(p^e)^{i \alpha} (-u(p^e))^{k}\sum_{m\leq \frac{N}{p^e}} \frac{b_p(m)}{m} = o(1). \]

It follows that \eqref{eq:tobeohof1} contributes $o(1)$. Putting this together with our earlier estimate, we find that
\[ \limsup_{N\to\infty} |T_1(N)| \le \frac{2}{B}. \]

Let us see why the $b_p$ have logarithmic mean zero. For this we apply  Proposition \ref{Prop GS}. Since $\sum_{q\text{ prime}}\frac{1}{q}$ diverges, it is enough to show that the partial sums of $\sum_{q \le x,\,\text{ prime}} \frac{\Re(b_p(q))}{q}$ remain bounded as $x\to\infty$. In fact, we will prove the stronger result that $\sum_{q\text{ prime}} \frac{\Re(b_p(q))}{q}$ converges.  

If $q\le p$, then $b_p(q)=0$, while if $q> p$,
\[ \Re(b_p(q))=\Re \left(\sigma(q)^{i \alpha} (u(q))^{k}\right)=\Re \left((q+1)^{i \alpha} \left(\frac{q}{q+1}\right)^{k}\right)=\left(\frac{q}{q+1}\right)^{k}\cos(\alpha\log(q+1)). \]

We have the following standard approximations:
$$
\left | \left (\frac{q}{q+1}\right )^k -1\right |\le \frac{k}{q+1}\quad\text{ and }\quad
\left |
\cos (\alpha \log (q+1))-
\cos (\alpha\log q)\right |
\le |\alpha| |\log (q+1) -\log q|\le \frac{|\alpha|}{q}.
$$
Hence, for $q>p$,
\[ \left|\Re(b_p(q)) -\frac{\cos(\alpha\log{q})}{q}\right| \le \frac{k}{q(q+1)} + \frac{|\alpha|}{q^2}. \]
The right hand side converges when summed on primes $q$. Thus, it suffices to show that
\[ \sum_{q\text{ prime}} \frac{\cos(\alpha\log{q})}{q} \quad\text{converges}. \]
But $\frac{\cos(\alpha\log{q})}{q} = \Re\frac{1}{q^{1+i\alpha}}$, and it is well known that $\sum_{q\,\text{ prime}} \frac{1}{q^{1+i\alpha}}$ converges for nonzero real $\alpha$ (see \cite[pp. 65-66]{Titchmarsh}, where this sum is related to the value of $\log \zeta(1+i\alpha)$).

\subsection{\bf Bounding the sum $T_2 (N)$.}

For the $n$ with a  special prime divisor $> B$, we can start with the trivial bound
\begin{equation*}
\begin{split}
\left |
\frac{1}{\log N}
\sum_{
\substack{{1<n\le N}\\
{\gamma (n)>B}}}
\frac{s(n)^{i\alpha}}{n}
\right |
&\le 
\frac{1}{\log N}
\sum_{
\substack{{1<n\le N}\\
{\gamma (n)>B}}}
\frac{1}{n}=: S.
\end{split}
\end{equation*}
 
Let $E = 10\log\log{B}$. We split the sum $S$ into three subsums: $S=S_1+S_2+S_3$, where the sum $S_1$ is over $p>B$, the sum $S_2$ is over $p\le B$ with $e\ge E$, and the sum 
$S_3$ is over $p<B$ for which $B\le p^e\le p^E$.

Now, the integers contributing to $S_1$ are $B$-sifted integers. We have
$$S_1  \le \frac{1}{\log N}\sum_{\substack{{n\le N}\\ {P^-(n)>B}}}\frac{1}{n} \le\frac{1}{\log{N}} \prod_{B < p \le N} (1-1/p)^{-1} \ll \frac{1}{\log{B}}.$$

 For $S_2$, we have
\begin{equation*}
\begin{split}
S_2 &\le \sum_{p\le B}
\sum_{E\le e\le \frac{\log N}{\log p}}
\sum_{m\le \frac{N}{p^e}}
\frac{1}{p^em\log N}\\
&\ll
\sum_{p\le B}\sum_{E\le e\le\frac{\log N}{\log p}}\frac{1}{p^e}\\
&\ll \sum_{p\le B}\frac{1}{p^E}\ll \frac{1}{2^E}\ll \frac{1}{\log B},
\end{split}
\end{equation*} 
recalling that $E=10\log\log B$.

It remains to bound the sum $S_3$, which corresponds to the prime powers $p^e>B$ with $p$ and $e$ not too large. Writing $n=p^e m$, we have:
\begin{equation*}
\begin{split}
S_3 &= \frac{1}{\log N}
\sum_{p\le B}
\sum_{\frac{\log B}{\log p}<e\le E}
\sum_{\substack{{m\le Np^{-e}}\\ {P^-(m)>p}}}
\frac{1}{p^em}\\
&\ll 
\sum_{p\le B}
\sum_{\frac{\log B}{\log p}<e\le E}\frac{1}{p^e\log p}\\
&\ll \sum_{p\le B}\frac{1}{B\log p}\ll \frac{1}{(\log B)^2}.
\end{split}
\end{equation*}

Putting together the estimates for $S_1, S_2$, and $S_3$, we conclude that $\limsup_{N\to\infty} |T_2(N)| \ll 1/\log{B}$. Combining this with our earlier result that $\limsup |T_1(N)| \le 2/B$, we find that
\[ \limsup |T(N)|  \ll \frac{1}{B} + \frac{1}{\log{B}}. \]
Since $B$ can be taken arbitrarily large, $T(N) = o(1)$. This completes the proof of Theorem \ref{thm:slogmean}.

\end{proof}

\begin{remark} An important step in the proof of Theorem \ref{th:Benford} is to apply Hal\'{a}sz's Theorem (Proposition \ref{Prop GS}) to prove that the functions $g_{\alpha,k} : n\mapsto \sigma (n)^{i\alpha} (n/\sigma (n))^k$ have logarithmic mean $0$ when $\alpha\in\R\setminus{0}$. This approach does not work in case of natural means because the functions $g_{\alpha, k}$ are $n^{i\alpha}$-pretentious in the sense of \cite{GranvilleSoundararajan}.
 
This suggests that $s(n)$ may not satisfy Benford's law with respect to natural density. In fact, we can prove this.

\begin{proposition} For every integer base $g\geq 2$, the sequence $s(n)$ is not Benford with respect to natural density in base $g$. \end{proposition}

\begin{proof} Fix a base $g \ge 2$. Recall that ``Benford in base $g$'' implies that $\{\log_g s(n)\}$ is uniformly distributed mod $1$ with respect to natural density. Roughly speaking, we show that for large $N$ the interval $I:=[g^N, g^N + g^{N-100}]$ contains many integers $n$ where $s(n)$ is slightly larger than $g^N$ -- more than that should be allowed if Benford's law holds for natural density. 

Call $n \in I$ \emph{convenient} if $n=6m$, where $m$ has no prime factor up to $g^{300}$. By inclusion-exclusion, the count of convenient $n$ is  
\[      \sim \frac{g^{N-100}}{6} \prod_{p \le g^{300}} (1-1/p) \] as $N\to\infty.$
Hence, for large $N$ this count exceeds
\[  \frac{g^{N-100}}{6} \cdot \frac{0.5}{\log(g^{300})} = \frac{g^{N-100}}{3600 \log g}. \]
Here, the product on $p$ is estimated using the explicit form of Mertens' theorem appearing as in \cite[eq.\ (3.25)]{rosser-schoenfeld}.

We claim that most convenient $n\in I$ have $s(n)$ very close to $n$, or equivalently, $\sigma(n)$ very close to $2n$. Since $6\mid n$, it is immediate that $\frac{s(n)}{n} \ge \frac{1}{n} \left(\frac{n}{6} + \frac{n}{3} + \frac{n}{2}\right) = 1$. To bound $\frac{s(n)}{n} = \frac{\sigma(n)}{n}-1$ from above, we use a first moment argument: 
\begin{align*}
  \sum_{\substack{n\in I\\n\,\text{convenient}}} \left(\frac{\sigma(n)}{n} - 2\right) &= \sum_{\substack{n\in I\\n\,\text{convenient}}} 2\left(\frac{\sigma(n/6)}{n/6}-1\right) 
        = 2 \sum_{\substack{n\in I\\n\,\text{convenient}}} \sum_{\substack{d \mid \frac{n}{6}\\d > 1}} \frac{1}{d} \le 2 \sum_{n\in I} \sum_{\substack{d\mid n\\d > g^{300}}} \frac{1}{d}\\
        &\le 2 \sum_{n \le 2g^N} \sum_{\substack{d\mid n\\d > g^{300}}} \frac{1}{d}
        \le 2 \sum_{d > g^{300}} \frac{1}{d} \sum_{\substack{n \le 2g^N\\d \mid n}} 1
        \le 4 g^N \sum_{d > g^{300}} \frac{1}{d^2}
        \le 4 g^{N-300}.\end{align*}
        
Hence, among the convenient $n \in I$, at most $4 g^{N-200}$ of them can satisfy $\frac{\sigma(n)}{n} \ge 2 + g^{-100}$. Thus, the number of convenient $n \in I$
with $s(n)/n < 1 + g^{-100}$ is at least \begin{equation}\label{eq:lb} 
   \frac{g^{N-100}}{3600 \log g} - 4 g^{N-200} = \frac{g^{N-100}}{3600 \log g} \left(1 - \frac{14400 \log g}{g^{100}}\right)
            > \frac{1}{2} \frac{g^{N-100}}{3600 \log g}. \end{equation}
Recalling the definition of $I$, all these convenient $n\in I$ with $s(n)/n < 1 + g^{-100}$ satisfy 
\[	s(n) \ge n \ge g^{N},\]
and also
\[ s(n) \le n(1+g^{-100}) \le g^{N} (1+g^{-100})^2. \]
Hence,
\[ N \le \log_g{s(n)} \le N + 2 \log_{g} (1+g^{-100}) \le N + 4 g^{-100}. \]
Therefore, the fractional part $\{\log_g s(n)\}$ belongs to the interval $[0, 4g^{-100}]$. 

If Benford's law holds for $s(n)$ with respect to natural density, then the number of $n \in I$ for which
$\{\log_g s(n)\}$ belongs to $[0, 4g^{-100}]$ is asymptotically $4g^{-100} \cdot g^{N-100} = 4g^{N-200}$ as $N\to\infty$. In particular, this number of such $n\in I$
is eventually smaller than $8g^{N-200}$. But that upper bound conflicts with the lower bound for the number of convenient $n\in I$ in \eqref{eq:lb} when $8g^{N-200} < g^{N-100}/(7200 \log g)$, or equivalently,
\[ 57600 \log g < g^{100}. \]
Since $g^{100}/\log g \ge g^{99} > 2^{99} > 57600$, we indeed have a contradiction.

\end{proof}

\end{remark}

\section{Ellipsephic $s(n)$ for composite $n$}
Throughout this section, let $x$ be large. Put 
	$y=y(x)=\exp (\sqrt{\log x})$. Recall that $\mathcal{W}_{\mathcal{D}}$ denotes the set of integers with no digit $a_0$ in their base $g$ expansion. Such integers are often called \emph{ellipsephic}.

\subsection{Preliminary lemmas}

Here we record the lemmas that we will use in our proof. The following is a special case of Lemma 2.3 in \cite{Pollack2014}.

	\begin{lemma}\label{Lemma 1-g} As $x\to\infty$, we have
	$$
	\Psi (x , y):=\#\{n \leqslant x: P^+(n) \leqslant y\} \ll x \exp \left(-\frac{\sqrt{\log x}}{2}\right).
	$$
	\end{lemma}
	
Next, we will need an upper bound for the count of integers divisible by a large square:
    
    \begin{lemma}\label{Lemma 2-g}
   If $x\geq z\geq 2$, then 
$$\#\left\{n \leqslant x:\, P^{+}(n)>z, P^{+}(n)^{2} \mid n\right\} \ll x/z, \,\text{ as }x\to\infty. $$
\end{lemma}
	\begin{proof} We have 
$$
\begin{aligned}
	\#\left\{n \leqslant x, P^{+}(n)>z, P^{+}(n)^{2} \mid n\right\} & \leqslant \sum_{z<p \leqslant \sqrt{x}} \sum_{\substack{n \leqslant x\\ p^2\mid n}} 1 \\
	& \leqslant \sum_{z<p} \frac{x}{p^{2}}\ll x/z.
\end{aligned}
$$
\end{proof} 
	

Finally, we will need  to count the integers $m$ in a certain range for which $s(m)$ is divisible by a fixed positive integer $q$.

\begin{lemma}\label{Lemma 3-g-old}

Let $\alpha\in (0,1)$, $\beta \in (0, \alpha/4)$, and let $q$ be a positive integer with $y^{\beta}<q\leq y $. For all $X\geq x^{\alpha}$, we have 
\begin{align*}
		\#\{x^{\alpha}< m\leq X:\, q\mid s(m)\}\ll \frac{\tau (q)}{\varphi (q)}X \log X,
	\end{align*} where $\tau$ is the divisor-counting function and $\varphi$ is Euler's totient function.
 \end{lemma}
 \begin{proof}

	Write $m=m_1P_1$ with $P_1=P^{+}(m)$. Let $y_2=y^2$.
 If $P_1\le y_2$, then all prime factors of $m$ are less than $y_2$ and thus the number of such $m\le X$ is less than
 $\Psi (X , y_2)$ which is $\ll X\exp (-\frac{\log X}{4\sqrt{\log x}}) \leq X\exp (-\alpha\frac{\sqrt{\log x}}{4}) =Xy^{-\alpha/4}$ by Lemma \ref{Lemma 1-g}.
 By Lemma \ref{Lemma 2-g}, we can also suppose that $(P_1,m_1)=1$.  It is thus sufficient to consider the case with 
	$P_1\nmid m_1$ and  $P_1\geq y_2$. Let $A$ be the corresponding set:
 \[
 A=\{ x^\alpha < m_1P_1\le X : q|s(m_1P_1),\, P^+ (m_1)<P_1,\ P_1>y_2\}.
 \]

 Then since $s(m)=P_1s(m_1)+\sigma(m_1)$, we have 
	\[P_1s(m_1)\equiv -\sigma(m_1)\bmod{q}. \]
	Let $q_1:=\gcd(s(m_1), q)$ and $q_2=q/q_1$. We fix such a factorization $q_1q_2$ of $q$.
	As $q_1$ divides both $s(m_1)$ and $q$, it also divides $\sigma(m_1)$. As a result, 
	\[q_1\mid\sigma(m_1)-s(m_1)=m_1. \]
	Moreover, $q_2=q/q_1$ is coprime to $s(m_1)/q_1$, and so the congruence
	\[P_1\frac{s(m_1)}{q_1}\equiv- \frac{\sigma(m_1)}{q_1}\bmod{q_2}\]
	puts $P_1$ in a uniquely determined residue class modulo $q_2$. Also, since $y_2\leqslant P_1\leq X/m_1$,  
	so $X/m_1\geq y_2$.\\
	By the Brun-Titchmarsh theorem, for a given $m_1$, the number of choices for $P_1$ is 
	\[\ll \frac{X/m_1}{\varphi(q_2)\log\frac{X}{m_1q_2}}\ll\frac{X/m_1}{\varphi(q_2)\log\frac{y_2}{y}} \]
	since, by assumption, $q_2 \leq q \leq y$, we have
	\[\frac{X}{m_1q_2}\geq \frac{X}{m_1 y}\geq \frac{y_2}{y} = \frac{y^2}{y} = y.\]

	Summing over $m_1\leq X/y_2$ that are multiples of $q_1$, we have (writing $m_1=q_1k_1$)
	
	\begin{align*}
 \sum_{k_1\leq X/y_2q_1}\frac{X/(q_1k_1)}{\varphi(q_2)\log y} \leq\frac{X}{q_1\varphi(q_2)\log y} \sum_{k_1\leq X/y_2q_1}\frac{1}{k_1}\ll\frac{X \log(X/y_2q_1)}{q_1\varphi(q_2)\log y} =\frac{X \log(X/y^2q_1)}{q_1\varphi(q_2)\log y}. 
 \end{align*}
	Now, summing over the choices of $q_1$ and $q_2$ yields for $\# A$:\\
	\begin{align*}
 \# A&\ll\frac{X}{\log y}\sum_{q_1q_2=q}\frac{\log(X/y^2q_1)}{q_1\varphi(q_2)} \\ &\ll\frac{X\log(X/y^2)}{\log y}\sum_{q_1q_2=q}\frac{1}{q_1\varphi(q_2)}. 
 \end{align*}
 
 As, the sum on $q_1,q_2$ is a convolution product
 \[ \sum_{q_1q_2=q}\frac{1}{q_1\varphi (q_2)} =\frac{1}{q}\sum_{q_2|q}\frac{q_2}{\varphi (q_2)}\le \frac{\tau (q)}{\varphi (q)},\]
 we deduce that \[\# A \ll \frac{\tau (q)}{\varphi (q)}{X\log X}.
\]

\end{proof}
 
\subsection{Proof of Theorem \ref{theorem-general}}
In the case where $g = 2$, the result trivially holds. Namely, since $a_0 \in \{1,...,g-1\}$, we must have $a_0 = 1$ when $g = 2$. Observe that there are no positive integers whose binary expansions contain no digit $1$. Since $s(n) > 0$ for all composite values of $n$, it follows that the count of composite $n \leq x$ for which $s(n)$ has no $1$ in its base-$2$ expansion is $0$. 

For the remainder of the proof, let $g \geq 3.$ We write $n=Pm$ where $P=P^+(n)$. By Lemmas \ref{Lemma 1-g} and \ref{Lemma 2-g}, we can suppose that $P\geq y$ and $(P,m)=1$.
The last assumption $(P,m)=1$ implies that 
\begin{equation}
\label{Pm}
s(n)=Ps(m)+\sigma (m). 
\end{equation}

First we handle the case $m\leqslant x^{\alpha_g}$ with $\alpha_g ={\frac{1}{2g^2\log^2 g}}\leq 1/4$. 
We consider
\[
E_1=\#\left\{
mP\leqslant x : m\leqslant x^{\alpha_g}  \text{ and } s(mP) \text{ has no $a_0$ in its base $g$ expansion}
\right\}.
\]
We will prove that
\begin{equation}\label{maj:E1}
E_1\ll x^{1-\frac{1}{2g^2\log ^2 g}}.
\end{equation}
First we have:
\[
E_1=\sum_{m\leqslant x^{\alpha_g}}
\#\left\{
P\leqslant x/m : s(mP) \text{ has no $a_0$ in its base $g$ expansion}
\right\}.
\]

Since $a:=s(n)\leqslant\sigma (n)\le n\tau (n)$, we have $\sigma(n)\ll n^{1+\varepsilon}$ for all $\varepsilon >0$. In particular for $mP\leqslant x$ with $P\le x$, $m\le x^{\alpha_g}$ as before,
\eqref{Pm} implies that $s(Pm)\leqslant x^{1+\frac{1}{g\log g}}$, for large enough $x$.
 If $m$ is fixed and  $a$ is given, there is at most one $P\leqslant x/m$
 such that $s(mP)=a$. In other words we have for any $m\leqslant x^{\alpha_g}$:
 \begin{equation*}
 \begin{split}
 &\#\{ P\le x/m : s(mP) \text{ has no $a_0$ in its base $g$ expansion}\}\\
 &\leqslant\#\left\{
 a\leqslant x^{1+\frac{1}{g\log g}}
 \text{ such that $a$ has no $a_0$ in its base $g$ expansion}
 \right\}.
 \end{split}
 \end{equation*}
Thus
\begin{equation}\label{small_m}
\begin{split}
E_1&\ll\sum_{m\leqslant x^{\alpha_g}}\#\{a\le x^{1+\frac{1}{g\log  g}}\ :a \text{ has no $a_0$ in its base $g$ expansion}\}.
\end{split}
\end{equation}

For any $X\ge 2$, the number of $a\leqslant X$ with no $a_0$ in base $g$ is $O((g-1)^N)$ where $N$ is the smallest integer such that $X\le g^N$  (ie. $N=\left\lceil \frac{\log X}{\log g}\right\rceil$).
Inserting this into \eqref{small_m} allows us to obtain

\begin{align*}
E_1&\leq \sum_{m\le x^{\alpha_g}}x^{(1+\frac{1}{g\log g})\frac{\log (g-1)}{\log g}}
\end{align*}
Note that $\frac{\log (g-1)}{\log g}=1+\frac{\log\left(1-\frac1g\right)}{\log g}$ which implies
\[\left(1+\frac{1}{g\log g}\right)\frac{\log (g-1)}{\log g} \leq 1-\frac{1}{g^2\log ^2g}.\]
Thus, recalling $\alpha_g =\frac{1}{2g^2\log ^2g}$, we obtain \eqref{maj:E1}
\begin{align*}
E_1&\ll  x^{\alpha_g+1-\frac{1}{g^2\log^2 g}} =x^{1-\frac{1}{2g^2\log ^2g}}.
\end{align*}

We now suppose that $m> x^{\alpha_g}$.
Let $k\in\mathbb{Z}$ such that $g^{k} \leqslant y<g^{k+1}$. Then  for $\mathcal{D}=\{0,1,\dots,g-1\}\setminus\{a_0\},$
		
		$$\#\{m P\leqslant x: m>x^{\alpha_g},\, P> y,\,P\nmid m,\,  s(mP) \text{ with no digit $a_0$} \}\leqslant \sum_{\substack{b<g^{k}\\b\in \mathcal{W}_{\mathcal {D}}}} \sum_{x^{\alpha_g}<m\leq x/y } \sum_{\substack{y<P\leq x/m\\P\nmid m\\s(mP) \equiv b \bmod {g^k} }} 1.
		$$
		
Noting \eqref{Pm}, $s(n)=s(mP)=Ps(m)+\sigma(m) \equiv b \bmod {g^{k}}$.

Let $d=\gcd(s(m), g^{k}). $
	We first consider the case where the gcd $d$ is small; that is, $d \leq y_{1}$ with $y_{1}= y^{2r\beta}$ and $\beta\leq \alpha_g/4$, where $r= \omega(g)$ with $\omega(g)$ denoting the number of distinct prime divisors of $g$. Then
	$$
	\frac{Ps(m)}{d}\equiv\frac{b-\sigma(m)}{d} \bmod{\frac{g^{k}}{d}}.
	$$
	Now, trivially, 
	$$\#\left\{y<P \leq \frac{x}{m}: 
	\frac{Ps(m)}{d}\equiv\frac{b-\sigma(m)}{d} \bmod{\frac{g^{k}}{d}}
	 \right\} \leqslant \frac{x/m}{g^{k} / d}.$$
Summing over $m$ and  $d$ yields 
	
\begin{align*}
	\sum_{x^{\alpha_g}<m\leq x/y} \sum_{1\le d \le y_{1}} \frac{x}{m} \frac{d}{g^{k}} \le& \frac{x}{g^k}\sum_{x^{\alpha_g}<m\leq x/y}\frac{1}{m} \sum_{1\le d \le y_{1}}d\\\le& 
 \frac{xy_1^2 \log{x/y}}{g^ky}\ll \frac{x}{y^{1/2}}.
\end{align*}
The final inequality follows from the fact that we have chosen $k$ such that $g^k\leqslant y<g^{k+1}$, thus $1/g^k\leqslant g/y\ll 1/y$.
Finally, summing over $b$, 
\begin{align*}
	\sum_{\substack{b\in \mathcal{W}_{\mathcal{D}}(g^k)}} \frac{x}{y^{1/2}}\leq \frac{x}{y^{1/2}}(g^k)^{\frac{\log \#\mathcal{D}}{\log g}}\ll x/y^{1/2-\frac{\log (g-1)}{\log g}}.
\end{align*}

Now, we consider the contribution from $m$, $P$, and $b$ when $(s(m), g^{k})>y_{1}$.
Let $V$ denote this contribution:
\[
V=\#\{ mP\le x : (s(m), g^k)\ge y_1,\ P>y, m>x^{\alpha_g},\ (P,m)=1\}.
\]
Since $mP\le x$ and $m>x^{\alpha_g}$, we have $P\le x^{1-\alpha_g}$.

In order to apply Lemma \ref{Lemma 3-g-old} with a convenient modulus $q$, 
 we use the prime factor decomposition of $g$:
$$g=p_1^{\gamma_1}\cdots p_r^{\gamma_r},$$
where $p_j$ are primes with $p_i\not = p_j$ if $i\not =j$ and $\gamma_1,\ldots \gamma_r\in\N$.
If $\gcd(s(m),g^k)\ge y_1$, then there exists $i\in\{ 1,\ldots ,r\}$ such that 
$\gcd(s(m), p_i^{k\gamma_i})>y_1^{1/r}$. This implies that $p_i^{\ell_i} |s(m)$ for some $\ell_i\in \mathbb{N}$ and we can choose $\ell_i$ with $p_i^{\ell_i}\le y_1^{1/r}<p_i^{\ell_i +1}$.
Thus we have
\[
V\ll \sum_{y<P<x^{1-\alpha_g}}\#\{ x^{\alpha_g} \le m\le x/P : p_i^{\ell _i}|s(m)\}.
\]

Now, we apply  Lemma \ref{Lemma 3-g-old} with $q=p_i^{\ell_i}>y_1^{1/(2r)}=y^{\beta_g}$ and $X=x/P$.
Since $$\frac{\tau (p_i^{\ell_i})}{\varphi (p_i^{\ell _i})} =\frac{1+\ell_i}{p_i^{\ell_i -1}(p_i-1)}\ll  y_1^{-1/r}\log (y_1)\ll y^{-2\beta_g}\log y\ll y^{-\beta_g}\log x,$$ we obtain

\[
V\ll \sum_{y<P<x^{1-\alpha_g}} \frac{x (\log x)^2}{P y^{\beta_g}}.
\]
Summing over $P$, and recalling that $y=\exp(\sqrt{\log x}),$ we obtain, for some constant $c>0$, 
\[
V\ll x\exp{(-c\sqrt{\log x})}.
\]

\appendix

\section{Extending Theorem \ref{thm: bcddt} to $g = 2$}\label{appendix}

In this appendix, we prove the omitted case $g=2$ of Theorem \ref{thm: bcddt}, thereby completing the proof of the theorem in every base.

Observe that, when $g = 2$, we must have $\mathcal{D} = \{1\},$ since there are no positive integers whose binary expansions contain only $0$'s. The positive integers whose binary expansions contain only the digit $1$ are precisely $2^{j} - 1$ for $j \geq 1.$ Thus, we want to count the values of $n \leq x$ for which $$s(n) = 2^j-1.$$ Let $$ K = \left\lceil \frac{2(\log\log x)^\gamma}{\log 2}\right\rceil.$$

If $2^K \mid \sigma(n)$, then $$n = \sigma(n) - s(n) \equiv 1 - 2^j \pmod{2^K}.$$ As $j$ varies, there are at most $K$ possible residue classes $\pmod{2^K}$ that $n$ can belong to: the classes $1 - 2^j$ for $1 \leq j < K$ and the class $1 \pmod{2^K}$ for all $j \geq K$. Consequently, this case contributes $$K \left(\frac{x}{2^K} + 1\right) \ll xe^{-(\log\log x)^\gamma}$$ values of $n$ to the count.

If $2^K \nmid \sigma(n)$, define $$R(n):= \#\{p: p \ \mathrm{odd \ prime}, \nu_p(n) \ \mathrm{odd}\}.$$ Since $$\sigma(n)= \sigma(2^{\nu_2(n)})\prod_{\substack{p\mid n\\ p\ \mathrm{odd}}}\sigma\!\left(p^{\nu_p(n)}\right),$$ where $\sigma(2^{\nu_2(n)})$ is odd, and $\sigma(p^{\nu_p(n)})$ is even whenever $\nu_p(n)$ is odd. So, each odd prime  divisor of $n$ with an odd $\nu_p(n)$, there are $R(n)$ many of them, contributes at least one factor of $2$ to $\sigma(n)$. Therefore, we have 
$$\nu_2(\sigma(n)) \ge R(n).$$ Since $2^K \nmid \sigma(n)$, we have $$\nu_2(\sigma(n)) < K,$$ hence $R(n) < K.$ Now, we estimate how many integers $n\leq x$ there are with $R(n) < K$. Fix $t \in (0, 1)$, say $t = 1/2.$ Then we have $$\mathbf{1}_{\{R(n) < K\}} \leq t^{R(n)-K}.$$ Indeed, if $R(n)<K$, then $R(n)-K<0$, and since $0<t<1$, we have $t^{R(n)-K} \geq 1$. Note that $(\frac{1}{2})^{R(n)}$ is multiplicative and, by the Selberg-Delange method (see \cite[Theorem II.5.3]{TenenbaumITAN22}), it has mean value $$\sum_{n \leq x} \left(\frac{1}{2}\right)^{R(n)} \ll \frac{x}{\sqrt{\log x}}.$$ Thus, we may conclude that $$\#\{n \leq x: 2^K \nmid \sigma(n)\} \ll 2^K \frac{x}{\sqrt{\log x}} \ll xe^{-(\log\log x)^\gamma},$$ since $\gamma < 1.$

\section*{Acknowledgements}
We would like to acknowledge ANR grant ANR-20-CE91-0006  ArithRand and the Utrecht University Mathematics Institute for providing funding for us to meet up in Nancy and Utrecht in order to work on this project. A portion of this manuscript was written while the fifth author was on sabbatical at the Max Planck Institute for Mathematics and the Centre de Recherches Mathématiques. She would like to thank both institutions for providing her with a pleasant working environment.

\end{document}